\documentclass{article}
\usepackage{amssymb}
\usepackage{amsfonts}
\usepackage{amsmath}
\usepackage{color}

\newtheorem{theorem}{Theorem}

\newtheorem{corollary}[theorem]{Corollary}

\newtheorem{definition}[theorem]{Definition}

\newtheorem{lemma}[theorem]{Lemma}

\newtheorem{proposition}[theorem]{Proposition}
\newtheorem{remark}[theorem]{Remark}

\newenvironment{proof}[1][Proof]{\noindent\textbf{#1.} }{\ \rule{0.5em}{0.5em}}

\begin{document}

\author{Adriano Da Silva\and
Departamento de Matematica, Universidad de Tarapac\'a, \\
Iquique, Chile\\
Email: adasilva@academicos.uta.cl \and Ey\"{u}p Kizil \and Department of
Mathematics, Yildiz Technical University, \\
34220 Istanbul, T\"{u}rkiye\\
Email: kizil@yildiz.edu.tr \and Victor Ayala \\
Department of Mathematics, University of Tarapac\'a\\
Arica, Chile\\
Email: vayala@academicos.uta.cl}
\title{Linear control systems on a 4D solvable Lie group used to model primary visual cortex $V1$}
\maketitle

\begin{abstract}
In this article, we study  linear control systems on a 4-dimensional solvable Lie group. Our motivation stems from the model introduced in \cite{baspinar}, which presents a precise geometric framework in which the primary visual cortex \( V1 \) is interpreted as a fiber bundle over the retinal plane \( M \) (identified with \( \mathbb{R}^{2} \)), with orientation \( \theta \in S^{1} \), spatial frequency \( \omega \in \mathbb{R}^{+} \), and phase \( \phi \in S^{1} \) as intrinsic parameters. For each fixed frequency \( \omega \), this model defines a Lie group \( G(\omega) = \mathbb{R}^{2} \times S^{1} \times S^{1} \), which we adopt in this work as the state space group \( G \) of our linear control system. We also present new results concerning controllability and characterize the control sets associated with this class of systems.

\end{abstract}

	
\section{Introduction}

The primary visual cortex, often called $V1$, is the primary cortical region
of the brain that receives, integrates, and processes visual information
relayed from the retinas. Different parts of the cerebral cortex (e.g. V1, V2, V3, V4, etc) analyze different features. Among them, V1 consists of neurons (i.e., simple cells) which
are locally sensitive to visual features such as orientation, spatial
frequency, phase etc. These simple cells of $V1$ are aligned in
hypercolumnar structure which was first observed in \cite{Hubel-Wiesel} and
further modelled via differential geometry by Hoffmann \cite{Hoffmann} as a
fiber bundle. Indeed, to each point of the visual or retinal plane $M$ a hypercolumn is assigned.

There are several models in the literature such that the visual cortex was
modelled as a Lie group. For example, in \cite{CittiSarti} a sub--Riemannian
structure on the Lie group $SE(2)=SO(2)\rtimes \mathbb{R}^{2}=S^{1}\times\mathbb{R}
^{2} $ of 2D proper motions provides a precise geometrical model
for the functional architecture of $V1$. In this model, Citti-Sarti
interpret the cortex as a fiber bundle over the retinal plane with only
orientation as an intrinsic variable. Note that appropriate interpretations
of $V1$ have been also treated in the literature with several requirements.
We constrain in this paper our attention to a recent model framework provided in \cite{baspinar} which generalizes that of Citti-Sarti in the sense that the cortex is interpreted as a fiber bundle over a 2-dimensional retinal plane where, in addition to orientation, spatial frequency and phase are introduced as extra intrinsic variables. The model provides at each frequency a 5-dimensional sub-Riemannian manifold $\mathcal{M}= S^{1}\times M\times 
\mathbb{R}^{+}\times S^{1}$ where $M$ stands for the retinal plane usually identified with $\mathbb{R}^{2}$. Note that the coordinate variables are $(\theta, x,y,\omega, s)$ where $\theta$ means orientation, $\omega \in \mathbb{R}^{+}$ a (spatial) frequency and $s$ spatial phase which takes values on $[0,2\pi]$. Hence, if $(x_0,y_0)$ and $\phi$ are given as fixed parameter values we should understand they are reference spatial position and phase values. At each fixed frequency $\omega $ the model provides
a Lie group 
$$G(\omega )=S^1\times \mathbb{R}^{2}\times S^{1}\simeq S^1\times \mathbb{R}^{2}\times\{\omega\}\times S^{1},$$ with a Lie
algebra generated by horizontal vector fields coming from the distribution
associated to the sub-Riemannian structure. Each point $(\theta, x,y,\phi)$ in $G(\omega )$ can be thought of as representing the preferred stimulus of a simple cell in $V1$.
\par It should be noted that $G(\omega )$ is a Lie group only if we fix $\omega \in \mathbb{R}^{+}$ in $\mathcal{M}
$ and it does not mean that $G(\omega)$ depends on $\omega $. The parameter $\omega $ is of course important for visual features but we study here only dynamics evolving on the groups $G(\omega)$ for each $\omega $ and not the variable $\omega$ itself. Since it is also necessary to distinguish any Lie group from the Lie group used to model $V1$ we will simply use in the rest of the exposition $G$ rather than the notation $G(\omega)$ and $L$ for any other Lie group. Analogously we fix the notation $\mathfrak{g}$ and $\mathfrak{l}$ to denote their corresponding Lie algebras.

Hence, we focus our interest on the Lie group $G= S^{1}\rtimes\mathbb{R}
^{2}\times S^{1}$ as the state space of linear control
systems. As we will show in the next section
this is a nonnilpotent solvable Lie group since descending sequence of its derived subalgabras
goes down to zero. Note that there exists a particular class between nilpotent and solvable Lie groups, which is the class of completely solvable ones. In our case, the group $G$ fails to be completely solvable since one of the
inner derivations admits complex eigenvalues.
\par We would like to stress in short the importance of solvable Lie groups (resp. Lie algebras) since the Lie group on which we consider our control systems belongs to the class of such Lie groups.  For one thing, solvable Lie algebras are exactly those that can be derived from semi-direct products. Moreover, they can be represented as triangular matrices. So, they are the most convenient Lie algebras to compute several mathematical objects and in particular their exponential might be easily determined.  On the other hand, any finite dimensional real Lie algebra $\mathfrak {l}$ can be expressed as the semi-direct product of a solvable ideal and a semi-simple (Lie)subalgebra. Thus, here is the departure point for more general results on arbitrary 4D Lie algebras. Note that the larger the Abelian component of a Lie algebra $\mathfrak {l}$, the more complex the algebra $\mathrm {Der}(\mathfrak{l})$ of derivations becomes.

The paper is divided into 4 sections. In Section 2, we introduce briefly the sub-Riemannian model of the primary visual cortex $V1$ which serves as the fundamental source for our motivation to consider linear control systems on the solvable Lie group $G$ used to model $V1$. In section 3, we start with emphasizing that the group $G$ is a pure solvable Lie group in the sense that it is neither nilpotent nor completely solvable and then  mainly constrain our attention to its factors since $G$ is the direct product of the Lie group $SE(2)$ of 2D proper motions of the plane and $S^{1}$. We also determine explicitly both derivations and automorphisms of the Lie algebra of $G$ to obtain linear vector fields and hence consider a linear control system whose state space is the group $G$. In the
last Section 4, we study controllability of a linear control system and give some concrete results. Controllability in the visual cortex model means how integral curves (parametrized with
different control inputs) transfer a neural activity that starts spreading from a certain
initial state in the brain to neurons in more distant states. We also determine control sets for this class of dynamics evolving on $G$.

\section{Set up}

We introduce very basic knowledge about the sub-Riemannian model of the primary visual cortex $V1$ and fix a set up for the rest of the manuscript. In particular, we specify our state space $G$ which is a Lie group
used for functional architecture of $V1$ at each (fixed) frequency $\omega $. Recall that by a sub-Riemannian manifold we
mean a $n$-dimensional differentiable real manifold $N$ equipped with a
nonintegrable distribution $\Delta$ of $\mathrm{rank}(\Delta)=$ $%
k<n$ and a sub-Riemannian metric $(g_{ij})_{p}:\Delta_{p}\times 
\Delta_{p}\rightarrow \mathbb{R}$ applied to vectors on subspaces $%
\Delta_{p}\subset T_{p}N$ at each $p\in N$. The distribution $\Delta\subset TN$ on the sub-Riemannian manifold $N$ is called horizontal
distribution and its (smooth) sections are called horizontal vector fields.
It follows by the well-known Chow Theorem that such a distribution $\Delta$ recovers the sub-Riemannian manifold $N$ (i.e. the ambient space) if the smallest Lie algebra $%
Lie(X_{1},\ldots ,X_{k})$ generated by the horizontal vector fields $%
X_{1},\ldots ,X_{k}$ is the whole tangent space $T_{p}N$ at each point $p\in
N$, that is, $Lie(X_{1},\ldots ,X_{k})(p)=T_{p}N$.

We start listing from \cite{baspinar} horizontal vector fields of the
5-dimensional sub-Riemannian manifold $\mathcal{M}=S^{1}\times\mathbb{R}^{2}\times
 \mathbb{R}^{+}\times S^{1}$ :%
\begin{eqnarray}
X_{1} &=&(\cos \theta )\partial /\partial x+(\sin \theta )\partial /\partial
y \\
X_{2} &=&\partial /\partial \theta \\
X_{3} &=&-(\sin \theta )\partial /\partial x+(\cos \theta )\partial
/\partial y+\omega \partial /\partial s \\
X_{4} &=&\partial /\partial \omega \text{.}
\end{eqnarray}%
These vector fields provide directions in which neural signals can propagate. Note that the variable $s$ stands for spatial phase. A straightforward computation involving partial differential operators implies
that 
\begin{eqnarray*}
\lbrack X_{1},X_{2}] &=&(\sin \theta )\partial /\partial x-(\cos \theta
)\partial /\partial y \\
\lbrack X_{2},X_{3}] &=&-(\cos \theta )\partial /\partial x-(\sin \theta
)\partial /\partial y \\
\lbrack X_{3},X_{4}] &=&-\partial /\partial s\text{.}
\end{eqnarray*}%
Note that $[X_{2},X_{3}]=-X_{1}$ and the rest of the brackets are all
trivial (i.e., $[X_{1},X_{3}]=[X_{1},X_{4}]=[X_{2},X_{4}]=0$). 

\par Let $T_{p}%
\mathcal{M}$ denotes the tangent space of $\mathcal{M}$ at $p\in \mathcal{M}$
and $V$ the vector space spanned by the vector fields $%
X_{1},X_{2},X_{3},X_{4}$ and $[X_{1},X_{2}]$. Hence $V$ might be interpreted as the space of directions along which neurons are connected in V1.
We mean by $V(p)\subset T_{p}%
\mathcal{M}$
the subspace obtained by evaluating the members of $V$ at $p$ and denote by $%
Lie(V)$ the smallest Lie algebra containing $V$. It follows that $T_{p}%
\mathcal{M}=Lie(V)$ which means the vector fields in $V$ are bracket
generating. As a matter of fact, the horizontal vector fields $%
X_{1},X_{2},X_{3},X_{4}$ and $[X_{1},X_{2}]$ provide the (long-range) connectivity between any two points of $\mathcal{M}$ via their integral curves (i.e. trajectories) defined on $\mathcal{M}$. 
Explicitly speaking, connectivity in the visual cortex is about how neurons and brain regions communicate to build a complete understanding of what we are seeing from detecting light and edges to recognizing objects and faces.
Note that by trajectory of a horizontal vector field we understand a curve $\gamma :[0,T]\rightarrow \mathcal{M}$ in $\mathcal{M}$
starting at the initial condition $\gamma (0)=(\theta
_{0}, v_{0},\omega _{0},\phi _{0})$ with $v_0=(x_0,y_0)$ such that $\gamma ^{\prime }(t)\in V(\gamma (t))$ at 
$\gamma (t)=(\theta (t), v(t),\omega (t),\phi (t))\in \mathcal{M}$.
\par We remember a bit from the terminology about receptive profiles of simple cells that form V1. It is known \cite {baspinar} that the hypercolumnar architecture of V1 is interpreted as a fiber bundle on the retinal plane. And a receptive profile of a simple cell depends on the position $v=(x,y)$ in the retinal plane $M$ and on the selected visual features $p=(\theta ,\omega ,\phi )$ such as orientation, frequency, phase, etc. The retina is the light-sensitive layer at the back of the eye and contains photoreceptors that convert light into electrical signals.
Moreover, the map defined in \cite {baspinar} by 
\begin{equation*}
A_{(\theta, v ,\phi )}:\left( 
\begin{array}{c}
\overline{x} \\ 
\overline{y} \\ 
\overline{s}%
\end{array}%
\right) \longmapsto \left( 
\begin{array}{c}
x \\ 
y \\ 
s%
\end{array}%
\right) =\left( 
\begin{array}{c}
x_{0} \\ 
y_{0} \\ 
\phi%
\end{array}%
\right) +\left( 
\begin{array}{ccc}
\cos \theta & -\sin \theta & 0 \\ 
\sin \theta & \cos \theta & 0 \\ 
0 & 0 & 1%
\end{array}%
\right) \left( 
\begin{array}{c}
\overline{x} \\ 
\overline{y} \\ 
\overline{s}%
\end{array}%
\right)
\end{equation*}%
is used to induce the Lie group
\begin{equation*}
G\simeq \{A_{(\theta, v, \phi )}:(\theta, v ,\phi
)\in S^{1}\times M\times  S^{1}\}
\end{equation*}%
by fixing $\omega$. The group operation is defined as follows:
\begin{equation} \label{grpmultip}
g_{1}g_{2}=\left(\theta
_{1}+\theta _{2}, v_1+R(\theta_1)v_2, \phi _{1}+\phi _{2}\right)
\end{equation}%
where $R (\theta )$ means the counter-clockwise rotation matrix by the angle $\theta $, and $g_{1}=((\theta _{1}, v_1),\phi _{1}),g_{2}=((\theta_2, v_2),\phi
_{2})\in G$ with $v_1=(x_{1},y_{1})$, $v_{2}=(x_{2},y_{2})$.

\begin{remark}
We note that the horizontal vector fields are left-invariant vector fields defined on $G$ so that one might consider a left-invariant dynamics on the group. However, we restrict our attention in this exposition to a wide class of control systems on $G$, namely, linear control systems that we formally define in Section \ref {secLCS}.
  
\end{remark}

\section{The group $G=SE(2)\times S^{1}$ and its Lie algebra }
The following standard notions (and more) from the theory of Lie algebras might be encountered in any textbook on Lie groups and their Lie algebras. We refer the reader for instance to a recent one \cite {Smartin}.
\begin{definition}
Let $\mathfrak{l}$ be a Lie algebra. By a Lie algebra derivation $\mathcal{D} $ of $\mathfrak{l}\,$ we mean a
linear mapping $\mathcal{D}:\mathfrak{l}\,\rightarrow \mathfrak{l}\,$ such that $%
\mathcal{D}[X,Y]=[\mathcal{D}X,Y]+[X,\mathcal{D}Y]$ for every $X,Y\in \mathfrak{l}$. 
We say that a linear mapping $\varphi:\mathfrak{l}\,\rightarrow \mathfrak{l}\,$ is an automorphism of $\mathfrak{l}$ if
$\varphi[X,Y]=[\varphi X,\varphi Y]$ for every $X,Y\in \mathfrak{l}$.
\end{definition}
In particular, the linear mapping $%
ad(X):\mathfrak{l}\rightarrow \mathfrak{l}$  defined by $%
ad(X)(Y)=[X,Y]$ for every $Y$ $\in \mathfrak{l}$ is called \emph{inner
derivation} of $\mathfrak{l}$.
As usual, we denote
by $\mathrm {Der}(\mathfrak{l})$ the set of all Lie algebra derivations of $\mathfrak{l}$ and by $\mathrm {Aut}(\mathfrak {l})$, the set of all Lie algebra automorphisms of $\mathfrak{l}$. It is known \cite {Smartin} that $\mathrm {Der}(\mathfrak{l})$ is the Lie algebra of the Lie group $\mathrm {Aut}(\mathfrak {l})$.

\begin{definition}
A Lie algebra $\mathfrak{l}$ is said to be 
\emph{solvable} if the descending sequence of derived subalgebras $\mathfrak{%
l}^{(k)}$ of $\mathfrak{l}$ goes down to zero for some integer $k$:%
\begin{equation*}
\mathfrak{l}\supset \mathfrak{l}^{(1)}\supset \mathfrak{l}^{(2)}\supset
\cdots \supset \mathfrak{l}^{(k)}={0}
\end{equation*}%
where $\mathfrak{l}^{(0)}=\mathfrak{l}$ and $\mathfrak{l}^{(k)}=[\mathfrak{l}%
^{(k-1)},\mathfrak{l}^{(k-1)}]$ for $k\geq 1$.
\end{definition}

 We say that a Lie group $L$ is solvable
if its Lie algebra $\mathfrak{l}$ is solvable. 
Remember that $L$ is said to be nilpotent if all iterated Lie brackets beyond some fixed uniform number iterations vanishes. Nilpotent Lie algebras form a subclass of solvable Lie algebras. However, there is an intermediate
class of Lie algebras between nilpotent and solvable ones which are
called \emph{completely solvable}. A Lie algebra $\mathfrak{l}$ is
completely solvable if $ad(X)$ for every $X\in \mathfrak{l}$ has only real eigenvalues.
\par We note that since the group $G$ of visual cortex  will be set as the state space of our dynamics later, it might help for further purposes to reveal some of its Lie group characteristic properties. Hence we find it convenient to start with the following elementary result.
\begin{proposition}
The Lie algebra $\mathfrak{g}$ of $G$ is spanned by the
vector fields $X_{1},X_{2},X_{3}$ and $[X_{1},X_{2}]$. It is non-nilpotent
solvable and not completely solvable Lie algebra.
\end{proposition}

\begin{proof}
An elementary computation regarding mixed partial derivatives shows that $%
[X_{1},[X_{1},X_{2}]]=[X_{3},[X_{1},X_{2}]]=0$ and $%
[X_{2},[X_{1},X_{2}]]=X_{1}$. It follows immediately from the formulas in (1)-(4) for horizontal vector fields that $X_{1},X_{2},X_{3}$
and $[X\,_{1},X_{2}]$ at each point $
(g,\phi )\in
G$ generate the tangent space 
\begin{equation*}
T_{(g,\phi )}G=Span\{X_{1},X_{2},X_{3},[X_{1},X_{2}]\}\text{.}
\end{equation*}%
On the other hand, that $[X_{1},[X_{1},X_{2}]]=0$ implies $\mathfrak{g}$ is indeed a solvable Lie algebra since the derived subalgebras
of $\mathfrak{g}$ stabilize at zero : 
\begin{eqnarray*}
\mathfrak{g}^{1} &=&[\mathfrak{g},\mathfrak{g}]=Span\{X_{1},[X_{1},X_{2}]\} \\
\mathfrak{g}^{2} &=&[\mathfrak{g}^{1},\mathfrak{g}^{1}
]=\{0\}\text{.}
\end{eqnarray*}%
It is easy to see that except $ad(X_{2}),$ all the inner derivations $%
ad(X_{1}),ad(X_{3})=-ad([X_{1},X_{2}])$ are nilpotent having zero eigenvalues. In fact, eigenvalues of $ad(X_{2})$ are $0,\pm \mathrm{%
i}$ (complex $\mathrm{i}$) and thus the Lie algebra $\mathfrak{g}$
fails to be completely solvable.
\end{proof}

We notice that it is more convenient for our purposes to work with another basis $\{E_{1},E_{2},E_{3},E_{4}\}$ rather than $\{X_{1},X_{2},X_{3},X_{4}\}$ which clarifies better what we intend to do in the next sections. Hence, let us consider the vectors on $\mathfrak{g}$ given by 
$$E_1=X_2, \hspace{.5cm} E_2=[X_1, X_2], \hspace{.5cm}  E_3=X_1, \hspace{.5cm} \mbox{ and }\hspace{.5cm} E_4=X_3+[X_1, X_2].$$
Note that we have especially assigned to $E_4$ the sum $X_3+[X_1, X_2]$ so that it commutes with the rest. It follows at once by the very definition of $X_1, X_2$ and $X_3$ above that one immediately obtains the nontrivial brackets as follows :
$$[E_1, E_2]=E_3,  \hspace{.5cm} [E_1, E_3]=-E_2. $$ Hence, the center of  $\mathfrak{g}$ is
$$\mathfrak{z}(\mathfrak{g})=\mathbb{R} E_4$$ and  $$\mathrm{span}\{E_1, E_2, E_3\}=\mathfrak{se}(2)$$ which implies that 
\begin{equation} \label{simboloisomorf}
\mathfrak{g}\simeq\mathfrak{se}(2)\times\mathfrak{z}(\mathfrak{g}) \hspace{.5cm}\mbox{ and }\hspace{.5cm} G\simeq SE(2)\times S^1
\end{equation}
where $SE(2)$ is the group of 2D proper  motions of $\mathbb{R}^2$ with the Lie algebra $\mathfrak{se}(2)$. Note that the symbol $\simeq$ stands for the isomorphisms that exist at both algebra and group levels. Thus, in order to better interpret the group $G$ we focus on both factors of it.

\subsection{Generalities on $SE(2)$}\label{subsecgeneralidades}
Following \cite{DaSilva}, the Lie algebra $\mathfrak{se}(2)$ is, up to isomorphism, given by the semi-direct product $\mathfrak{se}(2) = \mathbb{R} \times_J \mathbb{R}^2$ where $J=\left(\begin{array}{cc}
    0 & -1 \\
    1 & 0
\end{array}\right)$. The Lie bracket is defined by the relation
\[
[(1,0), (0, \eta)] = (0, J\eta), \quad \forall\, \eta \in \mathbb{R}^2.
\]

Similarly, the corresponding Lie group $SE(2)$ is given by the semi-direct product $SE(2) = S^1 \times_R \mathbb{R}^2$, where $R \colon S^1 \to \mathrm{GL}(2,\mathbb{R})$ is defined by \( R(\theta) = e^{\theta J} \), representing counter-clockwise rotation by an angle $\theta$. The group operation is given by
\begin{equation}
(\theta_1, v_1)(\theta_2, v_2) = (\theta_1 + \theta_2,\, v_1 + R(\theta_1)v_2).
\end{equation}

We now introduce the notion of a linear vector field on a connected Lie group, which will serve as the drift component of a linear control system later on.

\begin{definition} \label{deflinvecfield}
Let $L$ be a connected Lie group with Lie algebra $\mathfrak{l}$. A vector field $\mathcal{X}$ on $L$ is called a \emph{linear vector field} if its flow \( \{ \varphi_\tau \}_{\tau \in \mathbb{R}} \) forms a one-parameter subgroup of \( \mathrm{Aut}(L) \), the group of automorphisms of $L$.
\end{definition}

Note that in the Lie-group literature, linear vector fields on Lie groups are called infinitesimal automorphisms. There is another equivalent way to define linear vector fields on a connected Lie group $L$. Indeed, a vector field 
$\mathcal{X}$ on $L$ is linear if and only if 
\begin{equation} \label {normalizer}
[\mathcal{X}, Y] \in \mathfrak{l}, \hspace {0.5cm} \forall Y \in \mathfrak{l} 
\end{equation} 
and $\mathcal{X}$ has a singularity at the identity element. The expression (\ref{normalizer}) allows us to associate with a linear vector a unique derivation $\mathcal{D}:\mathfrak{l}\rightarrow \mathfrak{l}$ given by 
$\mathcal{D} Y:=[\mathcal{X}, Y](e)$. In particular, it holds that (see \cite[Proposition 2]{jou})
\begin{equation}
    \label{derivation}
    (d\varphi_t)_e=\mathrm{e}^{t\mathcal{D}}\hspace{.5cm}\mbox{ and }\hspace{.5cm}\varphi_t(\exp X)=\exp(\mathrm{t\mathcal{D}}X) 
\end{equation}

\par Using the previous setup, the authors in \cite{DaSilva} showed that linear and left-invariant vector fields on $SE(2)$ are given, respectively, as 
\begin{equation}
\mathcal{X}(\theta, v)=(0, Av+(1-R(\theta)J\xi))\hspace{.5cm}\mbox{ and }\hspace{.5cm} Y(\theta, v)=(\alpha, R(\theta)\eta)
\end{equation}
where $\alpha\in\mathbb{R}$, $\xi, \eta\in\mathbb{R}^2$ and $A\in\mathfrak{gl}(2, \mathbb{R})$ satisfies $AJ=JA$.
Moreover, derivations and automorphisms of the Lie algebra $\mathfrak{se}(2)$ and automorphisms of the corresponding Lie group group $SE(2)$ are given as (see \cite[Proposition 2.1 and 2.2]{DaSilva1})
\begin{equation} \label{derivse(2)}
\mathrm{Der}\left(\mathfrak{se}(2)\right)=\left\{\left(\begin{array}{cc}
    0 & 0 \\
   \xi  & A 
\end{array}\right):\hspace{.2cm} \xi\in\mathbb{R}^2, A\in\mathfrak{gl}(2, \mathbb{R}) \mbox{ with }AJ=JA\right\}
\end{equation}
\begin{equation} \label {autse(2)}
\mathrm{Aut}\left(\mathfrak{se}(2)\right)=\left\{\left(\begin{array}{cc}
    \pm 1 & 0 \\
   \eta  & P 
\end{array}\right):\hspace{.2cm} \eta\in\mathbb{R}^2, P\in\mathrm{GL}(2, \mathbb{R}) \mbox{ with }PJ=\pm JP \right\}
\end{equation}
and
\begin{equation} 
\mathrm{Aut}\left(SE(2)\right)=\left\{\begin{array}{cr}
    \varphi(\theta, v)=\left(\pm\theta, Pv+(1-R(\theta))J\eta\right): \;\eta\in\mathbb{R}^2 \hspace{.2cm} and \\ P\in\mathrm{GL}(2, \mathbb{R})
     \mbox{ with }PJ=\pm JP
\end{array}\right\}.
\end{equation}

\begin{remark}
    We will usually use the notation $\mathcal{X}=(\xi, A)$ to represent a linear vector field where $\xi\in\mathbb{R}^2$ and $A\in\mathfrak{gl}(2, \mathbb{R})$ are the ones that define $\mathcal{X}$.

\end{remark}

\subsection{Derivations and automorphisms of $\mathfrak{g}$ }

For further use we need to determine derivations and automorphisms of the Lie algebra $\mathfrak{g}$, as well as the automorphisms of $G$. 

Let us consider again the basis $\{E_1, E_2, E_3, E_4\}$ of $\mathfrak{g}$ such that $\mathfrak{se}(2)=\mathrm{span}\{E_1, E_2, E_3\}$ and $\mathfrak{z}(\mathfrak{g})=\mathbb{R} E_4$ and write now an element in $\mathfrak{g}$ as the pair $(X, a)$ with $X\in \mathfrak{se}(2)$ and $a\in \mathfrak{z}(\mathfrak{g})$. It follows that the Lie bracket rule in $\mathfrak{g}$ is nothing else than
\begin{equation} \label{bracketofg}
[(X, a), (Y, b)]=([X, Y],0)
\end{equation}
where $X, Y\in \mathfrak{se}(2), a, b\in\mathbb{R}.$
Moreover, as done in the previous section, when necessary, we write also $X=(t, v)\in\mathbb{R}\times\mathbb{R}^2$ for the elements in $\mathfrak{se}(2)$.

Hence let us start by obtaining an expression for a general derivation $\mathcal{D}$ of $\mathfrak{g}$ in these coordinates. We start by observing that, since the center $\mathfrak{z}(\mathfrak{g})$ is invariant by derivations, $\mathcal{D}$ is written, in blocks, as
\begin{equation}
    \mathcal{D}=\left(\begin{array}{cc}
    P & 0\\
    q^{\top} &  r
\end{array}\right), \hspace{.5cm} P\in\mathfrak{gl}(3, \mathbb{R}), q\in\mathbb{R}^3, r\in \mathbb{R}.
\end{equation}

It follows at once by (\ref{bracketofg}) that  for any $X, Y\in\mathfrak{se}(2)$ and $s, t\in\mathbb{R}$,
\begin{equation}
\mathcal{D}[(X, s), (Y, t)]=\mathcal{D}([X, Y], 0)=\left(P[X, Y], \langle q, [X, Y]\rangle\right)
\end{equation}
while
\begin{equation}
    [\mathcal{D}(X, s), (Y, t)]=\left[(PX, \langle q, X\rangle+rs), (Y, t)\right]=([PX, Y], 0)
\end{equation}
and

\begin{equation}
[(X, s), \mathcal{D}(Y, t)]=[(X, s), (PY, \langle q, Y\rangle+rt)]=([X, PY], 0).
\end{equation}
Hence $$\mathcal{D}[(X, s), (Y, t)]=[\mathcal{D}(X, s), (Y, t)]+[(X, s), \mathcal{D}(Y, t)],$$
which means $\mathcal{D}$ is a derivation of $\mathfrak{g}$, if and only if
\begin{equation}
P\in\mathrm{Der}(\mathfrak{se}(2))\hspace{.5cm}\mbox{ and }\hspace{.5cm} \langle q, [X, Y]\rangle=0,
\end{equation}
for every $X, Y\in\mathfrak{se}(2).$
By (\ref{derivse(2)}) in Subsection \ref{subsecgeneralidades}, we get that 
\begin{equation}
P=\left(\begin{array}{cc}
    0 & 0\\
    \xi & A
\end{array}\right), 
\end{equation}
where $\xi\in\mathbb{R}^2$ and $A\in\mathfrak{gl}(2, \mathbb{R}) \mbox{ with }AJ=JA.$
On the other hand, a straightforward calculation shows that for all $X, Y\in\mathfrak{se}(2)$
$$\langle q, [X, Y]\rangle=0\hspace{.5cm}\iff\hspace{.5cm} q=ae_1, \hspace{.5cm}e_1:=(1, 0, 0),$$
which results
\begin{equation} \label {derivforg}
\mathcal{D}=\left(\begin{array}{ccc}
    0 & 0 & 0  \\
    \xi & A & 0\\
    a & 0 & r
\end{array}\right),
\end{equation}
as the expression of a general derivation of $\mathfrak{g}$.

Regarding the automorphisms of $\mathfrak{g}$, the invariance of the center allows us to write $\varphi \in\mathrm{Aut}(\mathfrak{g})$ in blocks as
\begin{equation} \label{autoinblock}
\varphi =\left( 
\begin{array}{cc}
M & 0 \\ 
p^{\top}  & t%
\end{array}%
\right),
\end{equation}%
where $p\in \mathbb{R}^{3}$, $t\in \mathbb{R}^*$ and $M\in \mathrm{GL}(3,
\mathbb{R}).$
Remember that the requirement of being a Lie algebra automorphism here means
\begin{equation} \label{automorf}
\varphi \lbrack (X,a),(Y,b)]=[\varphi (X,a),\varphi (Y,b)],\quad \forall
(X,a),(Y,b)\in \mathfrak{g}.
\end{equation}%
Nonetheless, using the block form of $\varphi$ in (\ref{autoinblock}) one might re-write the left-hand side of the former equation as
\begin{equation}
\varphi [(X, a), (Y, b)]=\varphi ([X, Y], 0)=\left(M[X, Y], \langle p, [X, Y]\rangle\right).
\end{equation}
On the other hand, the right-hand side of the equation (\ref{automorf}) reads as
$$[\varphi(X, a), \varphi(Y, b)]=[(MX, \langle p, X\rangle), (MY, \langle p, Y\rangle)]=([MX, MY], 0),$$
and hence we get that $\varphi$ is a Lie algebra automorphism of $\mathfrak{se}(2)$ if and only if
\begin{equation}
M\in\mathrm{Aut}(\mathfrak{se}(2))\hspace{.5cm}\mbox{ and }\hspace{.5cm}  p=ae_1, a\in\mathbb
{R}.
\end{equation}
Finally, we conclude that 
$$\varphi =\left( 
\begin{array}{cc}
M & 0 \\ 
ae_1^{\top}  & t
\end{array}%
\right), \hspace{.3cm} t\in\mathbb{R}^*,  M\in\mathrm{Aut}(\mathfrak{se}(2)).$$
By the eq. (\ref{autse(2)}), the above matrix of $\varphi$ might be now written as a 3 by 3 matrix as : 
$$\varphi =\left( 
\begin{array}{ccc}
\pm 1 & 0 & 0 \\ 
\eta  & P & 0\\
a & 0 & t
\end{array}%
\right), \hspace{.2cm} a\in\mathbb{R}, t\in\mathbb{R}^*,  \eta\in\mathbb{R}^2$$
with $P\in \mathrm{GL}(2, \mathbb{R})$ is such that $PJ=\pm JP$. Hence we have proved the following

\begin{lemma}
\label{Lemma:der_aut}For the Lie algebra $\mathfrak{g}=\mathfrak{se}(2)\times\mathfrak{z}(\mathfrak{g})$, it holds that:  
$$\mathrm{Der}(\mathfrak{g})=\left\{\left(\begin{array}{ccc}
    0 & 0 & 0  \\
    \xi & A & 0\\
    a & 0 & r
\end{array}\right):\hspace{.2cm} \xi\in\mathbb{R}^2, a, r\in\mathbb{R}, \mbox{ and } A\in\mathfrak{gl}(2, \mathbb{R});  AJ=JA\right\},$$ 
and 
$$\mathrm{Aut}(\mathfrak{g})=\left\{\left(\begin{array}{ccc}
    \pm 1 & 0 & 0  \\
    \eta & P & 0\\
    a & 0 & t
\end{array}\right) :\hspace{.2cm} a\in\mathbb{R}, t\in\mathbb{R}^*,  \eta\in\mathbb{R}^2, \\ \mbox{ and } P\in \mathrm{GL}(2, \mathbb{R}); PJ=\pm  JP, \right\}.$$ 
\end{lemma}

\subsection{Group automorphisms of $G$ and its algebra}
We recall that the flow of a linear vector field $\mathcal{X}$ on $G$ is a 1-parameter subgroup of $\mathrm{Aut}(G)$ and hence it makes sense to find explicitly $\mathrm{Aut}(G)$. It is known that \cite {Smartin} this group is identified to a subgroup of $\mathrm{Aut}(\mathfrak{g})$ which we have already obtained in Lemma \ref {Lemma:der_aut}. Moreover, the Lie algebra $\mathfrak{%
aut}(G)$ of $\mathrm{Aut}(G)$ is a subalgebra of $\mathrm{Der}(\mathfrak{g})$. One might determine this subalgebra through
infinitesimal automorphisms as follows (see for instance \cite[Chapter 9]{Smartin} :
\begin{equation}
\mathfrak{aut}(G)=\{\mathcal{D}\in \mathrm{Der}(\mathfrak{g}):\mathcal{D}(X, a)=(0, 0),\forall X\in\mathbb{Z}E_1,  a\in\mathbb{Z}\}\text{.}
\end{equation}
We determine elements of $\mathfrak{%
aut}(G)$ as matrices in block in the following
\begin{proposition}
\label{der}
Let $G$ be the Lie group used to model $V1$. Then the Lie algebra $\mathfrak{aut}(G)$ of the Lie group $\mathrm{Aut}(G)$ is 
$$\mathfrak{aut}(G)=\left\{\left(\begin{array}{ccc}
    0 & 0 & 0  \\
    \xi & A & 0\\
    a & 0 & r
\end{array}\right): \xi\in\mathbb{R}{e}_2, r\in\mathbb{Z}, \mbox{ and } A\in\mathfrak{gl}(2, \mathbb{R}); \; AJ=JA\right\}.$$
\end{proposition}

\begin{proof}
In fact, by the expression obtained in (\ref {derivforg}) for the derivations $\mathcal{D}$ of $\mathrm{Der}(\mathfrak{g})$ we have that 
$$\mathcal{D}\left(\mathbb{Z}E_1\times\mathbb{Z}E_4\right)=(0, 0)\hspace{.5cm}\iff\hspace{.5cm} \langle\xi, {e}_1\rangle=0\mbox{ and }a\in\mathbb{Z},$$
which gives us $\xi=a {e}_2$, showing the result.\end{proof}

In virtue of the preceding Proposition it follows that $\mathrm{Aut}(G)$ is a 5-dimensional Lie group. However, we need to determine explicitly how a typical element of $\mathrm{Aut}(G)$ appears. For this purpose we give the following

\begin{proposition}
\label{auto}
Let $\varphi \in \mathrm{Aut}(G)$ be any Lie group automorphism of $G$. Then, there exists $a, t\in\mathbb{R}$ and $\widehat{\varphi}\in\mathrm{Aut}(SE(2))$ such that  $$\forall g=((\theta, v), \phi) \in G, \hspace{.5cm}\varphi(g)=\varphi((\theta, v), \phi)=\left(\widehat{\varphi}(\theta, v), t\phi+a\theta\right).$$ 
    
\end{proposition}

\begin{proof}
    Let $\phi\in\mathrm{Aut}(G)$. Since $(d\phi)_{(0, 0)}\in\mathrm{Aut}(\mathfrak{g})$, Lemma \ref{Lemma:der_aut} implies that we can write it as
    $$(d\phi)_{(0, 0)}=\left( 
\begin{array}{cc}
M & 0 \\ 
a e_1^{\top}  & t
\end{array}%
\right)$$
for some $M\in\mathrm{Aut}(SE(2)), t\in\mathbb{R}^*$ and $a\in\mathbb{R}$.
According to (\ref {autse(2)}) in Section 3.2, any element $M\in\mathrm{Aut}(\mathfrak{se}(2))$ determines a unique automorphism $\widehat{\varphi}\in \mathrm{Aut}(SE(2))$ such that $(d\widehat{\varphi})_{0}=M$. Define
$$\varphi((\theta, v), \phi)=\left(\widehat{\varphi}(\theta, v), t\phi+a\theta\right).$$
Since $G$ is a connected group, if we show that (i) $\varphi\in\mathrm{Aut}(G)$ and (ii) $(d\varphi)_{(0, 0)}=(d\phi)_{(0, 0)}$, then necessarily $\varphi=\phi$, \cite {Smartin}. The assertion (ii) follows immediately by the construction of $\varphi$ and it only remains to show the first one for which it is enough to show $\varphi(g_{1}g_{2})=\varphi(g_{1}) \varphi(g_{2})$ for every $g_{1}=((\theta _{1}, v_1),\phi _{1}),g_{2}=((\theta_2, v_2),\phi
_{2})\in G$.

On the other hand, we have that 
$$\varphi(g_{1}g_{2})=\varphi((\theta_1+\theta_2, v_1+R(\theta_1)v_2), \phi_1+\phi_2)$$
$$=\Bigl(\widehat{\varphi}(\theta_1+\theta_2, v_1+R(\theta_1)v_2), t(\phi_1+\phi_2)+a(\theta_1+\theta_2)\Bigr)$$
$$=\Bigl(\widehat{\varphi}((\theta_1, v_1)(\theta_2, v_2)), (t\phi_1+a\theta_1)+(t\phi_2+a\theta_2)\Bigr)$$
$$=\Bigr(\widehat{\varphi}(\theta_1, v_1)\widehat{\varphi}(\theta_2, v_2), (t\phi_1+a\theta_1 )+(t\phi_2+a\theta_2)\Bigr)$$
$$=\Bigr(\widehat{\varphi}(\theta_1, v_1), t\phi_1+a\theta_1 \Bigr)\Bigr(\widehat{\varphi}(\theta_2, v_2), t\phi_2+a\theta_2\Bigr)$$
$$=\varphi((\theta_1, v_1), \phi_1)\varphi((\theta_2, v_2), \phi_2)=\varphi(g_{1}) \varphi(g_{2}),$$
showing that $\varphi\in\mathrm{Aut}(G)$ and concluding the proof.
\end{proof}

\subsection{Linear and left-invariant vector fields on $G$}
Since $G$ is given by the cartesian product $G=SE(2)\times S^1$, any left-invariant vector field on $G$ can be written as
$$Z(g, \phi)=(Y(g), \beta), \hspace{.5cm}\mbox{ where } \hspace{.5cm} Y\in\mathfrak{se}(2), \beta\in\mathbb{R},$$

For linear vector fields, we have a similar decomposition

\begin{proposition}
\label{prop10}
Let $a\in\mathbb{R}$ and let $\widehat{\mathcal{X}}=(\xi, A)$ denote a linear vector field on $SE(2)$. Then, a linear vector field $\mathcal{X}$ on $G=SE(2)\times S^1$ and its associated derivation $\mathcal{D}$ are given, respectively, by:
\begin{equation}
\label{linearGomega}
\mathcal{X}((\theta, v), \phi)=\left(\widehat{\mathcal{X}}(\theta, v), a\theta\right)\hspace{.5cm}\mbox{ and }\hspace{.5cm} \mathcal{D}=\left(\begin{array}{ccc}
    0 & 0 & 0  \\
    \xi & A & 0\\
    a & 0 & 0
\end{array}\right).
\end{equation}
    
\end{proposition}

\begin{proof} If $\{\widehat{\varphi}_{\tau}\}_{\tau\in\mathbb{R}}$ is the flow of the linear vector field $\widehat{\mathcal{X}}$ on $SE(2)$, then 
$$(\tau, g)\in\mathbb{R}\times G\mapsto\varphi_{\tau}(g):=(\widehat{\varphi}_{\tau} (\theta, v), \phi+a\theta \tau)\in G,$$ for any $g=((\theta, v), \phi) \in G$
is, by Proposition \ref{auto}, a flow of automorphisms on $G$ since it satisfies the group property $\varphi_{\tau_1} \varphi_{\tau_2} = \varphi_{\tau_1+\tau_2}$. Indeed, $$\varphi_{\tau_1}((\varphi_{\tau_2} ((\theta, v), \phi)))=\varphi_{\tau_1}(\widehat{\varphi}_{\tau_2} (\theta, v), \phi+a\theta \tau_2)=$$
$$=(\widehat{\varphi}_{\tau_1}\widehat{\varphi}_{\tau_2}(\theta, v), (\phi+a\theta \tau_2)+a\theta \tau_1)=$$
$$=(\widehat{\varphi}_{\tau_1 + \tau_2}(\theta,v), \phi+a\theta (\tau_1 + \tau_2))=\varphi_{\tau_1 +\tau_2}((\theta,v),\phi)$$

   By construction, the vector field associated with such a flow is exactly $$\mathcal{X}((\theta, v), \phi)=\left(\widehat{\mathcal{X}}(\theta, v), a\theta\right),$$ showing that it is in fact a linear vector field.

In order to calculate the derivation of $\mathcal{X}$, we only have to show that the derivation of any linear vector field satisfies $r=0$ (see Lemma \ref{Lemma:der_aut}). However, if $\{\varphi_\tau\}_{\tau\in\mathbb{R}}\subset\mathrm{Aut}(G)$ is a flow of automorphisms of a linear vector field, the fact that the center 
$Z(G)$ of $G$ is invariant by automorphisms, implies that $\{(\varphi_\tau)|_{Z(G)}\}_{\tau\in\mathbb{R}}$ is a flow of automorphisms in $Z(G)$. However, $Z(G)=\{0\}\times S^1$ implies that $Z(G)$ is a torus and hence, $\mathrm{Aut}(Z(G))$ is discrete (see for instance \cite[Chapter 9]{Smartin}, implying that 
$$\varphi_\tau(0, \phi)=(0, \phi), \hspace{.5cm}\forall\phi\in S^1.$$
However, for any $\phi\in(0, 2\pi)$ we have that $(0, \phi)=\exp(0, \phi)$ and by (\ref{derivation}), we get
$$(0, \phi)=\varphi_\tau(0, \phi)=\varphi_\tau(\exp(0, \phi))=\exp(\mathrm{e}^{\tau\mathcal{D}}(0, \phi))=\exp(0, \mathrm{e}^{\tau r}\phi)=(0, \mathrm{e}^{\tau r}\phi),$$
implying that 
$$(1-\mathrm{e}^{\tau r})\phi\in 2\pi\mathbb{Z}, \hspace{.5cm}\forall \tau\in\mathbb{R},$$
which is only possible when $r=0$, showing the result.
\end{proof}

\begin{remark}
    As in the $SE(2)$ case, we will usually denote a linear vector field $\mathcal{X}$ on $G$ by $\mathcal{X}=(\xi, A, a)$.
\end{remark}

\section{Linear control systems on $G$: Control sets and controllability } \label {secLCS}

In this section we study the dynamics of LCSs on $G$ through their control sets. We start with a brief section containing the main definitions and concepts for general linear control systems on Lie groups.

\subsection{Linear control systems on Lie groups} \label {subsecLCS}

Let $L$ be a connected Lie group with Lie algebra $\mathfrak{l}$ identified with the set of left-invariant vector fields on $L$. In \cite{Aya-Tirao} the authors introduced the class of linear systems (abrev. LCS) on a connected Lie group $L$ governed by the following family of parametrized ordinary differential equations:

\begin{equation} \label {LCSemL}
\Sigma_L  : \hspace {.5cm}  \dot{x}=\mathcal{X}(x)+\sum_{j=1}^mu_j(t)Y^j(x), \hspace {.5cm} x \in L
\end{equation}
where the drift $\mathcal{X}$ is a linear vector field as in Definition (\ref {deflinvecfield}), $Y^1, \ldots, Y^m\in\mathfrak{l}$ and $u=(u_1, \ldots, u_m)\in\Omega$ with $\Omega\subset\mathbb{R}^m$ being a compact and convex subset such that the origin is in its interior.
\begin{remark}
We note that $\Sigma_L$ generalizes to arbitrary connected Lie groups one of the most relevant control systems on Euclidean space $\mathbb{R}^m$, namely, classical linear control systems on the commutative Lie group $\mathbb{R}^m$. Recall that a vector field $X$ on $\mathbb{R}^m$ is called linear if $X(x)=A(x)$ for all $x \in \mathbb{R}^m$ where $A:\mathbb{R}^m \rightarrow \mathbb{R}^m$ is a linear mapping. Hence, the drift vector field $\mathcal{X}$ in (\ref {LCSemL}) is a natural
extension to a Lie group $L$ of linear vector fields on the vector space $\mathbb{R}^m$ and, for this reason, we still call our system $\Sigma_L$ linear although it is, in several cases, non-linear.
\end{remark}

The class of linear systems on Lie groups might be generalized to homogeneous spaces and appear as models for more general class of control systems on manilfolds (For further details see \cite {jou}). 

\par For any $x\in L$ and ${\bf u}\in\mathcal{U}$, the solution $\tau\mapsto\Phi(\tau, x, {\bf u})$ of $\Sigma_L$ is complete in the sense that it is defined for every real $\tau$. Here, $\mathcal{U}$ is the set of admissible {\it control functions} given by
$$\mathcal{U}:=\{{\bf u}:\mathbb{R}\rightarrow\mathbb{R}^m: \;{\bf u}\;\mbox{ is a piecewise constant function with }  {\bf u}(\mathbb{R})\subset\Omega\}.$$

The {\it positive orbit} starting at $x$ is defined as
$$\mathcal{O}^{+}(x):=\{\Phi(\tau,x, {\bf u}): \tau\geq 0, {\bf u}\in \mathcal{U}\}.$$
Controllability is a powerful property of a control system. It means that given an initial condition $x$ and a desired final state, there exists a control input ${\bf u}$ such that the associated integral curve $\Phi(\tau,x, {\bf u})$ corresponding to the ordinary differential equation determined
by ${\bf u}$, transfers the initial point $x$ to the terminal state over a positive time interval.
The notion of control sets is relevant in any control system because it identifies
the regions in the state space where the challenging property of controllability holds.
\begin{definition}
Let $\Sigma_{L}$ denote the linear control system in (\ref{LCSemL}) such that $\mathcal{D}$ is the derivation associated with the system drift $\mathcal{X}$. We say that $\Sigma_{L}$  satisfies the {\it Lie algebra rank condition (abbrev. LARC)} if the Lie algebra $ \mathfrak{l}$  coincides with the smallest $\mathcal{D}$-invariant Lie subalgebra containing $Y^1, \ldots, Y^m$. 
\end{definition}

In particular, if the LARC is satisfied, $\mathrm{int}\mathcal{O}^{+}(x)$ is dense in $\mathcal{O}^{+}(x)$ for all $x\in L$. We recommend the classical book of V. Jurdjevic, \cite {jurdjevic}, to understand this density argument of the orbit.

Next, we introduce the concept of control sets.

\begin{definition}
\label{controlset}

A nonempty subset $\mathcal{C}$ of a Lie group $L$ is called a control set of the associated linear system $\Sigma_L$ if it satisfies the following properties:

\begin{itemize}
    \item [(i)] {\it (Weak invariance)} For every $x\in \mathcal{C}$, there exists a ${\bf u}\in\mathcal{U}$ such that $\Phi(\mathbb{R}^+,x,{\bf u})\subset \mathcal{C}$;
    \item [(ii)] {\it (Approximate controllability)}  $ \mathcal{C}\subset\mathrm{cl}\left(\mathcal{O}^{+}(x)\right)$ for every $x\in \mathcal{C}$;
    \item[(iii)] {\it (Maximality)} $\mathcal{C}$ is maximal, with relation to set inclusion, satisfying  properties (i) and (ii).
\end{itemize}
\end{definition}
In particular, when the whole state space $L$ is a control set, $\Sigma_L$ is said to be {\it controllable}. 
Note that the term ``approximate controllability" used in item (ii) of the preceding definition is due to the presence of the closure of the orbit.

\begin{remark}
\label{existence}
It is a standard fact (See Proposition 3.2.5 in \cite {Colonius}) that if $x_0\in\mathrm{int}\mathcal{O}^+(x_0)$, then there exists a control set $\mathcal{C}$ such that $x_0\in\mathrm{int}\mathcal{C}$. This fact, though standard, helps to assure the existence of control sets and will be used in our main result.
\end{remark}

Let us finish this section with some comments on conjugations of linear systems. Let $\Sigma_{H} $ be a LCS on a connected Lie group $H$
\begin{equation} \label {LCSemH}
\Sigma_{H} : \hspace {.5cm} \dot{z}(s) = \widehat{\mathcal{X}}(z(s)) + \sum_{j=1}^{m}u_j(s)\widehat{Y}^j(z(s)) 
\end{equation}
 with ${\bf u} = (u_1,\ldots, u_m)\in \mathcal{U}$. 

 \begin{definition} \label {defconjug}
 Let $\Sigma_L$ and $\Sigma_{H}$ denote two linear systems as in (\ref{LCSemL}) and (\ref{LCSemH}) on the corresponding Lie groups $L$ and $H$. We say that $\Sigma_L$ and $\Sigma_{H}$ are $f$-conjugated systems if there exists a surjective smooth map $f:L\rightarrow \ {H}$ such that  
for each $j \in \{1,\ldots m\}$ it holds
$$f_*\circ\mathcal{X}=\widehat{\mathcal{X}}\circ f\hspace{.5cm}\mbox{ and }\hspace{.5cm}\;\;\;f_*\circ Y^j=\widehat{Y}^j\circ f$$ where $f_*$ stands for the derivative of $f$. In particular, when $f$ is a diffeomorphism, $\Sigma_L$ and $\Sigma_{H}$ are called equivalent systems.
\end{definition}
The next result, whose proof can be found for instance in \cite[Proposition 2.4]{AOE}, relates control sets of conjugated systems. 

\begin{proposition}
\label{conjugation}
    Let $\Sigma_L$ and $\Sigma_H$ be $f$-conjugated systems. It holds:
\begin{enumerate}
    \item[1.] If $\mathcal{C}_L$ is a control set of $\Sigma_L$, there exists a control set $\mathcal{C}_H$ of $\Sigma_H$ such that $f(\mathcal{C}_L)\subset \mathcal{C}_H$;
    \item[2.] If for some $z_0\in \mathrm{int} \mathcal{C}_H$ it holds that $f^{-1}(z_0)\subset\mathrm{int} \mathcal{C}_L$, then $\mathcal{C}_L=f^{-1}(\mathcal{C}_H)$.
\end{enumerate}
\end{proposition}

\subsection{One-input LCS on $G$}
 The visual cortex might be interpreted as a dynamical system since the brain uses certain rules or laws (e.g. proximity, similarity, continuity, symmetry, etc) to organize what we see, which are carried out by the visual part of the brain (i.e. visual cortex). This is called perceptual organization and employs the sensory external inputs such as visual or retinal inputs. It follows that neurons in V1 and beyond are activated only when they receive these inputs from the retina and neural activity over time is governed by differential equations in this dynamical system.

Let us now consider our dynamical model on $G$ as a linear control system whose description depends on the factors $SE(2)$ and $S^1$ in its decomposition. Hence, let  
\begin{equation} \label {LCSemSE(2)}
		\Sigma_{SE(2)} :\hspace{.5cm} \dot{g}=\mathcal{X}(g)+uY(g), \hspace{.5cm}g=(\theta, v)\in SE(2)
	\end{equation}	
be a LCS on $SE(2)$ as defined in subsection (\ref {subsecLCS}) with $u\in \Omega$. Following \cite[Theorem 4.1]{DaSilva}, if $\det A\neq 0$, $\Sigma_{SE(2)}$ admits a unique control $\mathcal{C}_{SE(2)}$ set  with a nonempty interior. 

It follows that a (one-input) linear control system on $G$ might be given, in coordinates, as
\begin{equation} \label {LCSemG}
		\Sigma_{G}\:: \hspace{.5cm}\left\{\begin{array}{ll}
		     \dot{g}=\mathcal{X}(g)+uY(g)  \\
		   \dot{\phi}=a\theta +u\beta  
		\end{array},\right. \hspace{.5cm}(g, \phi)\in G = SE(2)\times S^1
	\end{equation}	
    where $u\in \Omega:=[u^-, u^+]$ such that $u^-<0<u^+$.
    \par 
Let $u$ be a control input,  $x=((\theta,v),\phi) \in G$ an initial point where $v$ is retinal position, $\theta$ is preferred orientation and $\phi$ is phase or motion direction. It follows that $\Phi(t,x,u)$ might be viewed as representing the firing rate or activation of neurons at that point.
\par It is not hard to see that the projection onto the first component  $$\pi:G\rightarrow SE(2), \hspace{.5cm} (g, \phi)\mapsto g$$
is a conjugation between $\Sigma_{G}$ and $\Sigma_{SE(2)}$ in the sense of Definition \ref {defconjug}.

\begin {remark} The input $u$ above might correspond to representations of visual stimuli where the sign and magnitude encode different aspects of the visual information. For example, negative (resp. positive) values of an input mean that the cell detects darkness (resp. brightness). This makes sense because our brain doesn’t respond to raw intensity but changes in light across space and time. In summary, the control input in $\Sigma_G$ might correspond to polarity of contrast (light vs dark), excitatory or inhibitory influence on a neuron, etc.
And the controllability issue here in this context means how integral curves (parametrized with different inputs) transfer a neural activity that starts spreading from a certain initial state in the brain to neurons in more distant states. 
\end{remark}

\par In what follows we will explore some conjugations in order to lift the controllability results from $SE(2)$ back to $G$.
The next result relates the LARC for both $\Sigma_{G}$ and $\Sigma_{SE(2)}$.

\begin{lemma}
    The linear system $\Sigma_{G}$ satisfies the LARC if and only if the projected system $\Sigma_{SE(2)}$ satisfies the LARC and $a^2+\beta^2\neq 0$.
\end{lemma}

\begin{proof}
    Let $\widehat{\mathcal{D}} \in \mathrm{Der}\left(\mathfrak{se}(2)\right)$ denote a derivation of $\mathfrak{se}(2)$ associated to $\widehat {\mathcal{X}}$. Thus we write by Proposition \ref {prop10} the derivation of $\mathfrak{g}$ associated to $\mathcal{X}$ as
    $$\mathcal{D}=\left(\begin{array}{ccc}
        \widehat{\mathcal{D}} & 0 \\
        ae_1^{\top} & 0
    \end{array}\right) \in \mathrm{Der}(\mathfrak{g}), \hspace{.5cm}e_1=(1, 0, 0)\in \mathbb{R}^3.$$ 
    
It follows for any $j\in\mathbb{N}$ and $(Y,\beta) \in \mathfrak {g}$ that 
$$\mathcal{D}^j(Y, \beta)=\left(\widehat{\mathcal{D}^j}Y, a\langle \widehat{\mathcal{D}}^{j-1}Y, e_1\rangle\right)$$
and
$$\left[\mathcal{D}^j(Y, \beta), \mathcal{D}^i(Y, \beta)\right]=\left(\left[\widehat{\mathcal{D}}^jY, \widehat{\mathcal{D}}^iY\right], 0\right).$$
However, by (\ref{derivse(2)}) it holds that
    $\widehat{\mathcal{D}}=\left(\begin{array}{ccc}
        0 & 0 \\
        \xi & A
    \end{array}\right)$ for some $\xi \in\mathbb{R}^2 \mbox{ and } A\in\mathfrak{gl}(2, \mathbb{R}) \mbox{ with }AJ=JA$, and hence 
$$\langle Y, e_1\rangle=a\alpha \hspace{.5cm}\mbox{ and }\hspace{.5cm}\langle \widehat{\mathcal{D}}^{j-1}Y, e_1\rangle=0, \hspace{.5cm}\forall j\geq2, $$ 
where we have used that $Y=(\alpha, \eta)$.
Hence, a vector field $(Z, \gamma)\in\mathfrak{g}$ belongs to the smallest $\mathcal{D}$-invariant Lie subalgebra containing $(Y, \beta)$ if and only if there are constants $\alpha_k, \alpha_{j,l}\in\mathbb{R}$, satisfying 
\begin{equation}  \label{forma}
    Z=\sum_{k=0}^2\alpha_k\widehat{D}^kY+\sum_{j, l}\alpha_{j, l}\left[\widehat{\mathcal{D}}^jY, \widehat{\mathcal{D}}^lY\right]\hspace{.5cm}\mbox{ and }\hspace{.5cm}\gamma=\alpha_0\beta+\alpha_1a\alpha.
\end{equation}

It follows by the former equation that one gets directly that, if $\Sigma_{G}$ satisfies the LARC, then $\Sigma_{SE(2)}$ also satisfies the LARC and $a^2+\beta^2\neq 0$. Reciprocally, let us assume that both conditions are satisfied. In this case the fact that $\Sigma_{SE(2)}$ in (\ref {LCSemSE(2)}) satisfies the LARC assures that, any $Z\in\mathfrak{se}(2)$ can be written as
\begin{equation*}
Z=\sum_{k=0}^2\alpha_k\widehat{D}^kY+\sum_{j, l}\alpha_{j, l}\left[\widehat{\mathcal{D}}^jY, \widehat{\mathcal{D}}^lY\right].
\end{equation*}
Therefore, any point $(Z, 0)\in\mathfrak{g}$ belongs to the smallest $\mathcal{D}$-invariant Lie subalgebra containing $(Y, \beta)$. Moreover, by equation (\ref{forma}), there exists $(X, \delta)\in\mathfrak{g}$, with $\delta\neq 0$,  that also belongs to such algebra. Since any $(Z, \gamma)$ can be written as
$$(Z, \gamma)=(Z-\gamma\delta^{-1}X, 0)+\gamma\delta^{-1}(X, \delta),$$
we get that the whole $\mathfrak{g}$ is contained in the smallest $\mathcal{D}$-invariant Lie subalgebra containing $(Y, \beta)$, concluding the proof.

\end{proof}



\bigskip



\subsection{The control set of LCSs in $G$}
Note that we have already commented at the earlier subsection about the notion of conjugation between control systems on Lie groups (resp. smooth manifolds). See Proposition \ref{conjugation}. Hence, before stating and proving our main result (i.e. Theorem \ref {theomain}), let us introduce the conjugation between $\Sigma_{G}$ and itself which is trivial but will be useful in the proof of our main result.

For any $\phi_1\in S^1$, the map 
    $$f_{\phi_1}:G\rightarrow G, \hspace{.5cm}f_{\phi_1}(g,\phi)=(g,\phi+\phi_1),$$ 
is basically the right translation in $G$ by the element $(0, \phi)$. Therefore, $f_{\phi}$ is in fact a diffeomorphism of $G$. On the other hand, the fact that 
$$\forall (g, \phi)\in G, \hspace{.5cm}(df_{\phi_1})_{(g, \phi)}=\mathrm{id}_{T_{(g, \phi)}G},$$
implies that $f_{\phi_1}$ is a conjugation. In particular, we get that 
\begin{equation}
    \label{formulaconjugation}
    \forall (g, \phi)\in G, \hspace{.5cm}f_{\phi_1}\left(\Phi(\tau, (g, \phi), u)\right)=\Phi(\tau, f_{\phi_1}(g, \phi), u),
\end{equation}
and, since $f_{\phi_1}$ is a diffeomorphism, 
$$\forall (g, \phi)\in G, \hspace{.5cm}\mathrm{int}\mathcal{O}^{+}(g, \phi)=f_{\phi}\left(\mathrm{int}\mathcal{O}^{+}(g, 0)\right).$$


Now we introduce our main result, namely

\begin{theorem}
\label{theomain}
    Let $\Sigma_{G}$ be a LCS on $G$ as in (\ref {LCSemG}) with linear vector field $\mathcal{X}=(\xi, A, a)$. If $\Sigma_{G}$ satisfies the LARC, it holds:
    \begin{itemize}
        \item [1.] If $\mathrm{det} A=0$, then $\Sigma_{G}$ admits an infinite number of control sets with empty interior;
        \item [2.] If $\det A\neq 0$, then $$\mathcal{C}_{SE(2)}\times S^1$$
        is the only control set with nonempty interior of $\Sigma_{G}$.
    \end{itemize}

\end{theorem}

\begin{proof}
    1. If $\det A=0$, the fact that $AJ=JA$ implies necessarily that $A=0$. Therefore, 
    $$\mathcal{X}((\theta, v), \phi)=(0, (1-R({\theta}))J\xi, a\theta)\hspace{.5cm}\mbox{ and }\hspace{.5cm} \mathcal{X}((0, v), \phi)=((0, 0), 0).$$
    Since any point in $P\in \{0\}\times\mathbb{R}^2\times S^1$ is a singularity of the drift, it has to be contained in a control set $\mathcal{C}_P$. Since the projection on the first coordinate $\pi:SE(2)\times S^1\rightarrow SE(2)$ conjugates $\Sigma_{G}$ and $\Sigma_{SE(2)}$, there exists by Proposition \ref{conjugation} above a unique control set $\mathcal{C}_{\pi(P)}$ of  $\Sigma_{SE(2)}$ satisfying $\pi(\mathcal{C}_P)\subset \mathcal{C}_{\pi(P)}$. However, as showed in \cite[Section 4.1]{DaSilva}, the control set $\mathcal{C}_{\pi(P)}$ is contained in the line $\{0\}\times (v+\mathbb{R}\xi)$ implying that
    $$\mathcal{C}_P\subset \{0\}\times (v+\mathbb{R}\xi)\times S^1,\hspace{.5cm} P=((0, v), \phi),$$
    which implies the assertion.

    \bigskip

    2. If $\det A\neq 0$, the only control set with nonempty interior of $\Sigma_{SE(2)}$ is $\mathcal{C}_{SE(2)}$. Let $g_0\in\mathrm{int}\mathcal{C}_{SE(2)}$. Since $\pi$ conjugates the systems, 
    $$\pi^{-1}\left(\mathrm{int} \mathcal{C}_{SE(2)}\right)\cap \mathrm{int}\mathcal{O}^+(g_0, \phi_0)\neq\emptyset$$
    for some $\phi_0\in S^1$. If $(g_1, \phi_1)\in \pi^{-1}\left(\mathrm{int} \mathcal{C}_{SE(2)}\right)\cap \mathrm{int}\mathcal{O}^+(g_0, \phi_0)$, the exact controllability in the interior $\mathrm{int}\mathcal{C}_{SE(2)}$ assures the existence of $\tau>0$ and $u\in\mathcal{U}$ such that $\Phi^1(\tau, g_1, u)=g_0$ and hence,
    $$(g_0, \phi_2)=\Phi(\tau, (g_1, \phi_1), u)\in\Phi_{\tau, u}\left(\mathrm{int}\mathcal{O}^+(g_0, \phi_0)\right)\subset \mathrm{int}\mathcal{O}^+(g_0, \phi_0).$$
 Since the elements of finite order are dense in $S^1$, perturbing $\phi_2$ inside $\mathrm{int}\mathcal{O}^+(g_0, \phi_0)$ allow us to obtain $\phi_3\in S^1$ such that 
 $$(g_0, \phi_3)\in \mathrm{int}\mathcal{O}^+(g_0, \phi_0)\hspace{.5cm}\mbox{ and }\hspace{.5cm}n(\phi_3-\phi_0)=0,$$
for some $n\in \mathbb{N
}$. Let then $\tau'>0$ and $u'\in \mathcal{U}$ such that $(g_0, \phi_3)=\Phi(\tau', (g_0, \phi_0), u')$ and extend $u'$ to a $\tau'$-periodic function. By concatenation and formula (\ref{formulaconjugation}), it holds that
$$\Phi((k+1)\tau', (g_0, \phi_0), u')=(g_0, k(\phi_3-\phi_0)+\phi_0),\hspace{.5cm}\forall k\in\mathbb{N}$$
which implies
$$(g_0, \phi_0)=\Phi((n+1)\tau', (g_0, \phi_0), u')\in \mathrm{int}\mathcal{O}^+(g_0, \phi_0).$$
Again by formula (\ref{conjugation}), we conclude that 
$$(g_0, \phi)=f_{\phi-\phi_0}((g_0, \phi_0))\in f_{\phi-\phi_0}\left(\mathrm{int}\mathcal{O}^+(g_0, \phi_0)\right)=\mathrm{int}\mathcal{O}^+(g_0, \phi), \hspace{.5cm}\forall \phi\in S^1.$$
Hence, for any $\phi\in S^1$, there exists a control set $\mathcal{C}_{\phi}$ of $\Sigma_{G}$ with $(g_0, \phi)\in\mathrm{int}\mathcal{C}_{\phi}$ (see Remark \ref{existence}). Since $S^1$ is compact, the fiber $\{g_0\}\times S^1$ is contained in a finite number of control sets $\mathcal{C}_{\phi_i}, i=1, \ldots, m$. Since $S^1$ is also connected, the control sets $\mathcal{C}_{\phi_i}$ have to intersect, implying actually that $\mathcal{C}:=\mathcal{C}_{\phi_i}$ for any $i=1, \ldots, m$. In particular, we get 
$$\pi^{-1}(g_0)=\{g_0\}\times S_1\subset\mathrm{int}\mathcal{C},$$
which again by Proposition \ref{conjugation} implies that 
$$\mathcal{C}_{SE(2)}\times S^1=\pi^{-1}(\mathcal{C}_{SE(2)})=\mathcal{C},$$
showing the equality. Uniqueness follows directly from the uniqueness of $\mathcal{C}_{SE(2)}$ and the previous equality.
\end{proof}

The preceding result together with \cite[Theorem 4.1]{DaSilva} implies the following.

\begin{corollary}
      A linear control system $\Sigma_{G}$ on $G$ with linear vector field $\mathcal{X}=(\xi, A, a)$ is controllable if and only if it satisfies the LARC, $\det A\neq 0$ and $\mathrm{tr} A=0$.
\end{corollary}

We would like to end our exposition commenting briefly about the set of singular points that might prevent completion in visual perception. Singularities usually refer to specific points in the functional architecture of the visual cortex where certain properties such as orientation preference change discontinuously or become undefined. These are particularly studied in terms of orientation maps in V1 which are often modeled using vector fields with singularities corresponding to zeros or discontinuities in these fields. This idea combines neuroscience, topology, and dynamical systems. Hence, the presence of our drift vector field $\mathcal{X}$ finds its concrete meaning in the dynamics we put on $G$ since it has a singularity at the identity point of $G$. Of course, none of the left-invariant vector fields (i.e. horizontal fields) on the system $\Sigma_G$ possess singular points. If we remember the equivalent definition of the drift $\mathcal{X}$ as in (\ref {normalizer})  we see that it does not belong to the horizontal layer but acts on horizontal vector fields so that it plays a structural role especially in scaling, symmetry, or as a modulatory field acting on the cortical geometry of the V1 structure

{\bf Acknowledgements} Adriano Da Silva's contribution to this work has been supported by Proyecto UTA Mayor Nº 4871-24.

{\bf Author Contributions} The authors contributed equally to the conceptualization of the research, the consequent analysis, and its writing up in the current form.

{\bf Declarations}
\par {\bf Conflict of interest} The authors declare no competing interests

\end{document}